\documentclass[12pt,final]{article}
\usepackage[margin=2.5cm,top=1.5cm]{geometry}
\usepackage[all,arc]{xy}
\usepackage{mathrsfs}
\usepackage{amsmath,amsthm,amssymb,enumerate}
\usepackage{color}

\newtheorem{definition}{Definition}[section]

\newtheorem{theorem}[definition]{Theorem}
\newtheorem{corollary}[definition]{Corollary}

\newtheorem{proposition}[definition]{Proposition}

\newtheorem{question}[definition]{Question}

\newtheorem{remark}[definition]{Remark}

\newcommand{\Romannum}[1]{\uppercase\expandafter{\romannumeral #1}}

\numberwithin{equation}{section}

\newcommand\keywordsname{Key words}
\newcommand\AMSname{AMS subject classifications}

\newenvironment{@abssec}[1]{%
     \if@twocolumn
       \section*{#1}%
     \else
       \vspace{.05in}\footnotesize
       \parindent .2in
         {\upshape\bfseries #1. }\ignorespaces
     \fi}
     {\if@twocolumn\else\par\vspace{.1in}\fi}

\begin{document}

\title{Some remarks on $P$, $P_0$, $B$ and $B_0$ tensors
\footnote{P. Yuan's research is supported by the NSF of China (Grant No. 11271142) and
the Guangdong Provincial Natural Science Foundation(Grant No. S2012010009942),
L. You's research is supported by the Zhujiang Technology New Star Foundation of
Guangzhou (Grant No. 2011J2200090) and Program on International Cooperation and Innovation, Department of Education,
Guangdong Province (Grant No.2012gjhz0007).}}
\author{Pingzhi Yuan\footnote{{\it{Corresponding author:\;}}yuanpz@scnu.edu.cn.}, \qquad  Lihua You \footnote{{\it{Email address:\;}}ylhua@scnu.edu.cn.}}
\vskip.2cm
\date{{\small
School of Mathematical Sciences, South China Normal University,\\
Guangzhou, 510631, P.R. China\\
}}
\maketitle

\begin{abstract}
Recently, Song and Qi extended the concept of $P$, $P_0$ and $B$ matrices to $P$, $P_0$, $B$ and $B_0$ tensors,
 obtained some properties about these tensors, and  proposed many questions for further research. In this paper,
 we answer three questions mentioned as above and obtain  further results about  $P$, $P_0$, $B$ and $B_0$ tensors.

\end{abstract}

\section{Introduction}

\hskip.6cm Since the work of  Lim \cite{Li05} and Qi \cite{Qi05}, the study  of tensors (and hypergraphs) and their various applications has attracted much attention and interest.

A real $m$th  order $n$-dimensional tensor (hypermatrix)  $\mathbb{A}= (a_{i_1i_2\ldots i_m})$ is a multi-array of real entries
 $a_{i_1i_2\ldots i_m}$,  where $i_j\in [n]=\{1,2,\ldots, n\}$ for $j\in [m]$.
Denote the set of all real $m$th order $n$-dimensional tensors by $T_{m,n}$. Then $T_{m,n}$ is a linear space of dimension $n^m$.
Let $\mathbb{A}= (a_{i_1i_2\ldots i_m})\in T_{m,n}$, if the entries $a_{i_1i_2\ldots i_m}$ are invariant under any permutation of their indices,
then $\mathbb{A}$ is called a symmetric tensor. Denote the set of all real $m$th order $n$-dimensional symmetric tensors by $S_{m,n}$.
Then $S_{m,n}$ is a linear subspace of $T_{m,n}$.

Let $\mathbb{A}= (a_{i_1i_2\ldots i_m})\in T_{m,n}$ and $x\in \mathfrak{R}^n$. Then $\mathbb{A}x^m$ is a homogeneous polynomial of degree $m$, defined by
$$\mathbb{A}x^m=\sum\limits_{i_1,\ldots, i_m=1}^{n}a_{i_1\ldots i_m}x_{i_1}\ldots x_{i_m}.$$

A tensor $\mathbb{A}\in T_{m,n}$ is called positive semi-definite if for any vector $x\in \mathfrak{R}^n$, $\mathbb{A}x^m\geq 0$,
and $\mathbb{A}\in T_{m,n}$  is called positive definite if for any nonzero vector $x\in \mathfrak{R}^n$, $\mathbb{A}x^m> 0$.
Clearly, if $m$ is odd, there is no nontrivial positive semi-definite tensors.

It is well known that $P$  matrices  and $P_0$ matrices which were first studied systematically by Fiedler and Pt$\acute{a}$k \cite{F62}
have a long history and wide applications in linear complementarity problems,
variational inequalities and nonlinear complementarity problems and so on.
In 2001, Pena proposed and studied $B$ matrices \cite{Pena01}, obtained many nice properties and applications of such matrices \cite{Pena01, Pena03},
 and proved that the class of $B$ matrices is a subclass of $P$ matrices  \cite{Pena01}.
Recently, Song and Qi extended the concept of $P$, $P_0$ and $B$ matrices to $P$, $P_0$, $B$ and $B_0$ tensors,
 obtained some nice properties about these tensors in \cite{SQ2, SQ1}, and they also proposed many questions for further research.

Let $\mathbb{A}= (a_{i_1i_2\ldots i_m})\in T_{m,n}$ and $x\in C^n$. Then $\mathbb{A}x^{m-1}$ is a vector in $C^n$ with its $i$th components as
\begin{equation}\label{eq1}
(\mathbb{A}x^{m-1})_i=\sum\limits_{i_2,\ldots, i_m=1}^{n}a_{ii_2\ldots i_m}x_{i_2}\ldots x_{i_m}.
\end{equation}

\noindent for $i\in [n]$.

 \begin{definition}{\rm (\cite{SQ1}, Definition 2)}\label{defn11}
Let $\mathbb{A}= (a_{i_1i_2\ldots i_m})\in T_{m,n}$.  We say that $\mathbb{A}$ is

{\rm (a) } a $P_0$ tensor if for any nonzero vector $x\in \mathfrak{R}^n$, there exists $i\in [n]$ such that $x_i\not=0$ and $x_i(Ax^{m-1})_i\geq 0$;

{\rm (b) } a $P$ tensor if for any nonzero vector $x\in \mathfrak{R}^n$, $\max\limits_{i\in [n]}x_i(Ax^{m-1})_i> 0$.
 \end{definition}

 \begin{definition}{\rm (\cite{SQ1}, Definition 3)}\label{defn12}
Let $\mathbb{A}= (a_{i_1i_2\ldots i_m})\in T_{m,n}$. We say that $\mathbb{A}$ is a $B$ tensor if for all $i\in [n]$,
$\sum\limits_{i_2, \ldots, i_m=1}^na_{ii_2\ldots i_m}>0$
and
$\frac{1}{n^{m-1}}(\sum\limits_{i_2, \ldots, i_m=1}^na_{ii_2\ldots i_m})>a_{ij_2\ldots j_m} \mbox { for all } (j_2,  \ldots, j_m)$ $\not=(i,\ldots, i).$
 We say that $\mathbb{A}$ is a $B_0$ tensor if for all $i\in [n]$,
$\sum\limits_{i_2, \ldots, i_m=1}^na_{ii_2\ldots i_m}\geq 0$
and
$$\frac{1}{n^{m-1}}(\sum\limits_{i_2, \ldots, i_m=1}^na_{ii_2\ldots i_m})\geq a_{ij_2\ldots j_m} \mbox { for all } (j_2,  \ldots, j_m)\not=(i,\ldots, i).$$
 \end{definition}

 In \cite{Qi05}, the definitions of eigenvalues, $H$-eigenvalues, $E$-eigenvalues and $Z$-eigenvalues of tensors in $S_{m,n}$ were proposed.
In \cite{SQ1}, Song and Qi extended there concepts to tensors in $T_{m,n}$,   and they obtain some properties about $P$ and $P_0$ tensors as follows.

 \begin{theorem}{\rm (\cite{SQ1}, Theorem 1)}\label{thm13}
 Let $\mathbb{A}\in T_{m,n}$ be a $P (P_0)$ tensor. Then

{\rm (1) } When $m$  is even, all of its $H$-eigenvalues  and $Z$-eigenvalues of $\mathbb{A}$ are positive (nonnegative).

 {\rm (2) } A symmetric tensor is a $P (P_0)$  tensor if and only if it is positive (semi-)definite.

 {\rm (3) } There does not exist an odd order symmetric $P$ tensor.

{\rm  (4) } If an odd order non-symmetric $P$ tensor exists, then it has no $Z$-eigenvalues.

{\rm  (5) }  An odd order $P_0$ tensor has no nonzero $Z$-eigenvalues.
 \end{theorem}

Based on the result (3) of Theorem \ref{thm13}, Song and Qi proposed the following Question \ref{ques14}.

\begin{question}{\rm (\cite{SQ1}, Question 1)}\label{ques14}
 Is there an odd order non-symmetric $P$ tensor? Is there an odd order nonzero non-symmetric $P_0$ tensor?
 \end{question}

It is well known that each $B$ matrix is a $P$ matrix, namely,  when $m=2$, each $B$ tensor is a $P$ tensor.
In \cite{SQ1}, the authors noted that when $m$ is odd, in general, a $B (B_0)$ tensor is not a $P (P_0)$ tensor.
For example, let $a_{i\ldots i} = 1$ and $a_{i_1\ldots i_m} = 0$ otherwise.
Then $\mathbb{A} = (a_{i_1\ldots i_m})$ is the identity tensor.
But when $m$ is odd, the identity tensor is a $B$ tensor,  not a $P$ or $P_0$ tensor. Thus they proposed the following Question \ref{ques15}.

\begin{question}{\rm (\cite{SQ1}, Question 4)}\label{ques15}
When $m \geq 4$ and is even, is a $B (B_0)$ tensor a $P (P_0)$ tensor?
 \end{question}

 In \cite{SQ2}, the authors proved that an even order symmetric $B$ tensor is positive definite,
 and  an even order symmetric $B_0$ tensor is positive semi-definite.   Combining the result (2) of Theorem \ref{thm13},
 we know that an even order symmetric $B$ tensor is a $P$ tensor and an even order symmetric $B_0$ tensor is a $P_0$ tensor.
Further, the author proposed the following question.

\begin{question}{\rm (\cite{SQ2}, Question 2)}\label{ques16}
Can we show that an even order non-symmetric $B$ tensor is a $P$ tensor and an even order non-symmetric $B_0$ tensor is a $P_0$ tensor?
\end{question}

 In this paper, we answer the above three questions and obtain  further results about  $P$, $P_0$, $B$ and $B_0$ tensors.

\section{Some remarks on $P$ and $P_0$ tensors }
\hskip.6cm In this section, we  show that there does not exist an odd order $P$ tensor
and there exist odd  order nonzero non-symmetric $P_0$  tensors, thus we give the answer of Question \ref{ques14}.

\begin{proposition}\label{prop21}
There does not exist an odd order  $P$ tensor.
 \end{proposition}
\begin{proof}
Let $m$ be odd  and $\mathbb{A}=(a_{i_1i_2\ldots i_m})\in T_{m, n}$ be an odd order $P$ tensor.
 For some given $k\in[n]$, we take nonzero vector $x=(0, \ldots, 0, u, 0, \ldots, 0)=(x_1, \ldots, x_{k-1}, x_k, x_{k+1}, \ldots, x_n)$,
where $x_k=u (\not=0)$ and $x_i=0$ for $i\ne k$. Since for any $i\in [n]$, by (\ref{eq1}), we have

$$(\mathbb{A}x^{m-1})_i=\sum\limits_{i_2, \ldots, i_m=1}^na_{ii_2\ldots i_m}x_{i_2}\ldots x_{i_m}=a_{ik\ldots k}u^{m-1},$$

\noindent then $$x_i(\mathbb{A}x^{m-1})_i=\left\{
   \begin{array}{cc}
   0, & \mbox {if }  i\ne k; \\
    a_{k k \ldots k}u^{m},  & \mbox {if } i=k.
   \end{array} \right.$$
\noindent

Since $\mathbb{A}$ is a $P$ tensor,  we have $\max\limits_{i\in [n]}x_i(\mathbb{A}x^{m-1})_i=a_{k k \ldots k}u^{m}>0$,
then $a_{k k \ldots k}\not=0$.
 Noting that $m$ is odd, thus $a_{kk \ldots k}>0$ when  we take $u>0$  and  $a_{kk \ldots k}<0$ when  we take $u<0$ by
 $a_{k k \ldots k}u^{m}>0$, it is a contradiction. Thus there does not exist an odd order  $P$ tensor.
\end{proof}

In \cite{SQ1},  the authors proved that there does not exist an odd order symmetric $P $ tensor by the theory of eigenvalues of tensors,
now we can obtain the result by proposition \ref{prop21} directly.
\begin{corollary}\label{cor22}
{\rm (1) (\cite{SQ1}, (3) of Theorem 1) } There does not exist an odd order symmetric $P $ tensor.

{\rm (2) } There does not exist an odd order non-symmetric $P$ tensor.
\end{corollary}

\begin{proposition}\label{prop23}
There exist odd order nonzero non-symmetric $P_0$ tensors.
 \end{proposition}
\begin{proof}
Let $m$ be odd. Now we show the result by the following two cases.

\noindent {\bf Case 1: }  $n$ is even.

Let $n=2k\ge 2$. Take $a_{i(k+i)i\ldots i}=1,$ $ a_{(k+i)i\ldots i}=-2$ for any $i\in [k]$ and
$a_{i_1\ldots i_m}=0$ for others $(i_1, \ldots, i_m)\in[n]^{m}$. Clearly, $\mathbb{A}=(a_{i_1\ldots i_m})$ is an odd order nonzero non-symmetric tensor. Now we show that $\mathbb{A}$ is a $P_0$ tensor.

For any nonzero vector $x=(x_1,x_2, \ldots, x_{2k})^T\in \mathfrak{R}^{2k}$,   noting that
$$x_i(\mathbb{A}x^{m-1})_i=x_i^{m-1}x_{k+i}, \quad  x_{k+i}(\mathbb{A}x^{m-1})_{k+i}=-2x_i^{m-1}x_{k+i} \mbox { for any } i=1, \ldots, k, $$

\noindent then by $x\not=0$, there exists some $j\in [k]$ such that $x_j\not=0$ or $x_{k+j}\not=0$,
and

  $$x_j(\mathbb{A}x^{m-1})_j\geq 0 \mbox { or } x_{k+j}(\mathbb{A}x^{m-1})_{k+j}\geq 0.$$

\noindent  Thus  $\mathbb{A}$ is a $P_0$ tensor by the definition of $P_0$ tensor.

\noindent {\bf Case 2: }  $n$ is odd.

Let $n=2k+1\ge 3$.  Take $a_{i(k+i)i\ldots i}=1, a_{(k+i)ii\ldots i}=-2$ for $i\in [k-1]$, $a_{k(2k)(2k+1)k\ldots k}=1, a_{(2k)(2k+1)k\ldots k}=-2, a_{(2k+1)(2k)k\ldots k}=4$,  and
$a_{i_1\ldots i_m}=0$ for others $(i_1, \ldots, i_m)\in[n]^{m}$. Clearly, $\mathbb{A}=(a_{i_1\ldots i_m})$ is an odd order nonzero non-symmetric tensor. Now we show that $\mathbb{A}$ is a $P_0$ tensor.

For any nonzero vector $x=(x_1,x_2, \ldots, x_{2k+1})^T\in \mathfrak{R}^{2k+1}$,  noting that
$$x_i(\mathbb{A}x^{m-1})_i=x_i^{m-1}x_{k+i}, \quad x_{k+i}(\mathbb{A}x^{m-1})_{k+i}=-2x_i^{m-1}x_{k+i} \mbox { for any }  i=1, \ldots, k-1, $$

\noindent and
$x_k(\mathbb{A}x^{m-1})_k=x_k^{m-2}x_{2k}x_{2k+1}, \hskip.2cm  x_{2k}(\mathbb{A}x^{m-1})_{2k}=-2x_k^{m-2}x_{2k}x_{2k+1},
 \hskip.2cm x_{2k+1}(\mathbb{A}x^{m-1})_{2k+1}=4x_k^{m-2}x_{2k}x_{2k+1}, $
then by $x\not=0$, we can complete the proof by the following two cases.

\noindent {\bf Subcase 2.1: } there exists some $j\in [k-1]$ such that $x_j\not=0$ or $x_{k+j}\not=0$.

 It is obvious that $\mathbb{A}$ is a $P_0$ tensor by $x_j(\mathbb{A}x^{m-1})_j\geq 0 \mbox { or } x_{k+j}(\mathbb{A}x^{m-1})_{k+j}\geq 0$ and
 the definition of $P_0$ tensor.

\noindent {\bf Subcase 2.2: }  $x_k\not=0$ or $x_{2k}\not=0$ or $x_{2k+1}\not=0$.

 It is obvious that $\mathbb{A}$ is a $P_0$ tensor by the definition of $P_0$ tensor and the facts that $x_k(\mathbb{A}x^{m-1})_k\geq 0 \mbox { or } x_{2k}(\mathbb{A}x^{m-1})_{2k}\geq 0 $  or $x_{2k+1}(\mathbb{A}x^{m-1})_{2k+1}\geq 0$.

Combining the above arguments, $\mathbb{A}$ is an odd order nonzero non-symmetric $P_0$ tensor.
\end{proof}

\begin{remark}\label{rem24}
Noting that whether ``$m$ is odd" or not does not change the fact that   $\mathbb{A}$ is a nonzero non-symmetric $P_0$ tensor.
Thus the proof of Proposition \ref{prop23} also give an example of an even order nonzero non-symmetric $P_0$ tensor.
\end{remark}

Clearly, Propositions \ref{prop21} and \ref{prop23} answer Question \ref{ques14} (Question 1 in \cite{SQ1}).

\section{The relationship between $B(B_0)$ and $P(P_0)$ tensors}

\hskip.6cm In this section, we answer Questions \ref{ques15} and  \ref{ques16},  obtain some further results about $B (B_0)$ tensors and $P (P_0)$ tensors.

\begin{proposition}\label{prop31}
When $m \geq 4$ and is even, in general, a  $B (B_0)$ tensor is not a $P (P_0)$ tensor.
 \end{proposition}
\begin{proof}
Let $m(\ge4)$ be even and $n(\ge 3)$ be a positive integer, $\mathbb{A}=(a_{i_1\ldots i_m})$ where $a_{11\ldots1}=h+1>1, a_{1i_2\ldots i_m}=h>0,$
 $a_{ii\ldots i}=1$ and   $a_{i11i\ldots i}=b$ for any $i\in\{2, \ldots, n\}$,
 and $a_{i_1i_2\ldots i_m}=0$ for others.

 It is easy to check that  $\mathbb{A}$ is an even order non-symmetric $B$ tensor when $b=-\frac{1}{2}$ and
 $\mathbb{A}$ is an even order non-symmetric $B_0$ tensor when $b=-1$   by the definitions of $B (B_0) $ tensor.

  Now for any nonzero vector $x=(x_1, \ldots, x_n)^T\in \mathfrak{R}^n$, we have
$$x_i\left(\mathbb{A}x^{m-1}\right)_i=\left\{
   \begin{array}{ll}
  x_1^m+hx_1(x_1+\ldots+ x_n)^{m-1}, & \mbox { if } i=1;  \\
   x^m_i+bx_1^2x_i^{m-2},   & \mbox { if } i\in \{2, \ldots, n\}.
   \end{array} \right.$$
It is easy to see that there exist real numbers $h, x_1, \ldots, x_n$ with $x_i\left(\mathbb{A}x^{m-1}\right)_i<0$ for all $i\in[n]$ and $x_i\not=0$.
For example,   we take $x_1=-3, x_2=\ldots =x_n=2$, $h>(\frac{3}{2n-5})^{m-1}$,
we have $x_i\left(\mathbb{A}x^{m-1}\right)_i<0$ for all $i\in [n]$ and $x_i\not=0$.
Then $\mathbb{A}$ is not a $P (P_0)$ tensor by the definitions of $P (P_0)$ tensor.
\end{proof}

\begin{remark}\label{rem32}
Clearly, Proposition \ref{prop31}  answers Questions \ref{ques15} and \ref{ques16} (Question 4 in \cite{SQ1} and Question 2 in \cite{SQ2}).
Furthermore, by Proposition \ref{prop31} and some known results, we know that a $B (B_0)$ tensor is not a $P (P_0)$ tensor in general.
That is to say, the cases of $m\geq 3$ (hypermatrices) is different from the case of $m=2$ (matrices) completely.
\end{remark}

Let $\mathbb{A} \in T_{m,n}$. If all of the off-diagonal entries of $\mathbb{A}$ are non-positive, i.e., $a_{i_1\ldots i_m}\leq 0$ when $(i_1, \ldots, i_m)\not=(i, \ldots, i)$, then $\mathbb{A}$ is called a $Z$ tensor (\cite{ZQZ14}).  In the following, we will show if an even order  $Z$ tensor $\mathbb{A}$ is a $B (B_0)$ tensor, then  $\mathbb{A}$ is a $P (P_0)$ tensor.

We call $\mathbb{A}$ is  diagonally dominated if for all $i\in [n]$,
$$a_{i\ldots i}\geq \sum\{|a_{ii_2\ldots i_m}|: (i_2,\ldots, i_m)\not=(i, \ldots, i), i_j\in [n], j=2, \ldots, m\};$$
\noindent and we call $\mathbb{A}$ is strictly diagonally dominated if for all $i\in [n]$,
$$a_{i\ldots i}> \sum\{|a_{ii_2\ldots i_m}|: (i_2,\ldots, i_m)\not=(i, \ldots, i), i_j\in [n], j=2, \ldots, m\}.$$

It was proved in \cite{ZQZ14} that a diagonally dominated $Z$ tensor is an $M$ tensor, and a
strictly diagonally dominated $Z$ tensor is a strong $M$ tensor (\cite{DQW13,ZQZ14}), where strong $M$ tensors are called nonsingular tensors (\cite{DQW13}),
and Laplacian tensors of uniform hypergraphs which is defined as a natural extension of Laplacian matrices of graphs are M tensors.

In \cite{SQ1}, the authors give the properties of a $B (B_0)$ tensor under the condition that it is a $Z$ tensor as follows.

\begin{theorem}{\rm (\cite{SQ1}, Theorem 8)}\label{thm33}
 Let $\mathbb{A} = (a_{i_1i_2\ldots i_m})\in T_{m,n}$  be a $Z$ tensor. Then the following properties are equivalent:

{\rm (i) } $\mathbb{A}$  is $B(B_0)$ tensor;

{\rm (ii) } for each $i\in [n]$, $\sum\limits_{i_2,\ldots, i_m=1}^{n}a_{ii_2\ldots i_m}$ is positive (nonnegative);

{\rm (iii) } $\mathbb{A}$ is strictly diagonally dominant (diagonally dominated).
\end{theorem}

\begin{theorem} \label{thm34}
Let $m$ be even, $\mathbb{A} \in T_{m,n}$ be a strictly diagonally dominated (diagonally dominated)  tensor,
then $\mathbb{A}$ is a $P (P_0)$ tensor.
\end{theorem}

\begin{proof}
Let  $\mathbb{A} \in T_{m,n}$ be a strictly diagonally dominated   tensor.
Then for any nonzero vector $0\ne x\in\mathfrak{R}^n$, without loss of generality,
 we assume that $|x_1|=\max\limits_{i\in[n]}\{|x_i|\}$, then $|x_1|\ne0$.
 By the definition of strictly diagonally dominated and $m$ is even, we have
$$x_1\left(\mathbb{A}x^{m-1}\right)_1=a_{1\ldots1}|x_1|^m+
\sum\limits_{i_2, \ldots, i_m=1,(i_2,\ldots, i_m)\not=(1,\ldots,1)}^na_{1i_2\ldots i_m}x_1x_{i_2}\ldots x_{i_m}$$

\hskip3.8cm $\geq\left(a_{1\ldots1}-\sum\limits_{i_2, \ldots, i_m=1, (i_2,\ldots, i_m)\not=(1,\ldots,1)}^n|a_{1i_2\ldots i_m}|\right)|x_1|^m$

\hskip3.8cm $>0.$

\noindent Thus $\mathbb{A}$ is a $P$ tensor by the definition of $P$ tensor.

If $\mathbb{A} \in T_{m,n}$ is a  diagonally dominated  tensor, the proof is similar, we omit it.
\end{proof}

By Theorem \ref{thm34} and the result (2) of Theorem \ref{thm13}, we obtain the following result immediately.
\begin{corollary}{\rm (\cite{SQ2}, Theorem 3)}\label{cor35}
Let $m$ be even and $\mathbb{A}\in S_{m,n}$. If $\mathbb{A}$ is diagonally dominated,
then $\mathbb{A}$ is positive semi-definite. If $\mathbb{A}$ is strictly diagonally dominated,
then $\mathbb{A}$ is positive definite.
\end{corollary}

By Theorems \ref{thm33} and \ref{thm34}, we obtain the following result directly.
\begin{theorem} \label{thm36}
Let $m$ be even, $\mathbb{A} \in T_{m,n}$ be a $Z$  tensor.
If $\mathbb{A}$ is a $B (B_0)$ tensor, then  $\mathbb{A}$ is a $P (P_0)$ tensor.
\end{theorem}

\begin{remark}
Let $\mathbb{A} \in T_{m,n}$ be a $Z$  tensor. If $\mathbb{A}$  is a $B$ tensor, in general,  $\mathbb{A}$  is not positive definite.
For example, let $b>0, h>0$ and $\mathbb{A}=(a_{i_1\ldots i_m})$, where $a_{11\ldots 1}=b+1,$  $a_{12\ldots 2}=-b$,
$a_{ii\ldots i}=1$ for all $i\in \{2, \ldots, n\}$,  and $a_{i_1i_2\ldots i_m}=0$ for others.
Clearly, $\mathbb{A}$ is a strictly diagonally dominated  $Z$ tensor, namely,
$\mathbb{A}$ is both a   $Z$ tensor and a $B$ tensor by Theorem \ref{thm33}.
Then for any vector $x=(x_1,\ldots, x_n)^T\in \mathfrak{R}_n$, we have $\mathbb{A}x^m=\sum\limits_{i=1}^nx_i^m+bx_1^m-bx_1x_2^{m-1}$.
 Now we take $x_1=1, x_2=\ldots =x_n=h$,
thus $\mathbb{A}x^m=(b+1)+(n-1)h^m-bh^{m-1}$.  Clearly, there exist positive real numbers $b$ and $h$ such that $\mathbb{A}x^m<0$,
then $\mathbb{A}$  is not positive definite. But by Theorem \ref{thm36}, $\mathbb{A}$ is a $P$ tensor,
it implies that a $P$ tensor, in general, is not positive definite.
\end{remark}

\section{Decomposition of $B (B_0) $ tensors}
\hskip.6cm In \cite{SQ2}, the authors  show that a symmetric $B (B_0)$ tensor can always be decomposed to the sum of a strictly
diagonally dominated (diagonally dominated) symmetric $Z$ tensor and several positive multiples of partially
all one tensors. When the order is even, this result implies that a symmetric $B (B_0)$ tensor is
positive definite (positive semi-definite), namely, a symmetric $P (P_0)$ tensor.
In this section, we weaken the condition ``symmetric" and obtain the generalized results by using  some similar technique  in \cite{SQ2}.

This concept of ``a principal sub-tensor'' was first introduced and used in \cite{Qi05} for symmetric tensor as follows.

\begin{definition}{\rm(\cite{Qi05})}\label{defn41}
A tensor $\mathbb{C}\in T_{m,r}$ is called a principal sub-tensor of a tensor $\mathbb{A} = (a_{i_1\ldots i_m})\in T_{m,n}$ $(1 ¡Ü\leq r\leq n)$
 if there is a set $J$ that composed of $r$ elements in $[n]$ such that
$$\mathbb{C} =(a_{i_1\ldots i_m}), \mbox { for all } i_1, i_2, \ldots, i_m\in J.$$
\end{definition}

For convenient, we denote by $\mathbb{A}^J_r$ the principal sub-tensor of a tensor $\mathbb{A}\in T_{m,n}$
such that the entries of  $\mathbb{A}^J_r$ are indexed by $J\subseteq [n]$ with $|J| = r$ for $1\leq  r \leq  n$.
It is clear that when $r=1$, the principal sub-tensors are just the diagonal entries.

\begin{definition}{\rm(\cite{SQ2})}\label{defn42}
Suppose that $A\in S_{m,n}$ has a principal sub-tensor $A^J_r$ with $J\subseteq[n]$ with $|J| = r
(1 \le r \le n)$ such that all the entries of $A^J_r$ are one, and all the other entries of $A$ are
zero. Then $A$ is called a partially all one tensor, and denoted by $\mathcal{E}^J$.
If $J = [n]$, then we denote $\mathcal{E}^J$ simply by $\mathcal{E}$ and call it an all one tensor.
\end{definition}

When $m$ is even, if we denote by $x_J$ the $r$-dimensional sub-vector  of a vector $x\in\mathfrak{R}^n$,
with the components of $x_J$ indexed by $J$, then for any $x=(x_1,\ldots, x_n)^T \in\mathfrak{R}^n$,
we have \begin{equation}\label{eq41}
\mathcal{E}^Jx^m =\left(\sum\{x_j : j\in J\}\right)^m\ge0.\end{equation}
Thus  an even order partially all one tensor is positive semi-definite.

Let $S_k$ be the set of all permutations $\sigma$ of the integers $1,2,\ldots k$, and symbol
$$\delta_{i_1\ldots i_m}=\left\{\begin{array}{cc}
                                     1, & \mbox{ if } i_1=\ldots =i_m; \\
                                     0, & \mbox{ otherwise. }
                                   \end{array}\right.$$

\begin{theorem} \label{thm43}
Let $\mathbb{A} = (a_{i_1\ldots i_m}) \in T_{m,n}$ be a tensor such that $a_{i_1\ldots i_m}=a_{\sigma(i_1)\ldots \sigma(i_m)}$ for any $\sigma\in S_m$ whenever $a_{i_1\ldots i_m}>0$ and $\delta_{i_1\ldots i_m}=0$.

{\rm (1) } If  $\mathbb{A}$ is a $B_0$ tensor, then either $\mathbb{A}$ is a diagonally dominated  $Z$ tensor itself, or we have
\begin{equation}\label{eq42}
\mathbb{A} =\mathbb{M}+\sum\limits_{k=1}^s h_k\mathcal{E}^{J_k},
 \end{equation}
where $\mathbb{M}$ is a diagonally dominated  $Z$ tensor, $s$ is a positive integer, $h_k > 0$ and
$J_k\subseteq [n]$, for $k = 1, \ldots, s$, and $J_s\subsetneqq J_{s-1}\subsetneqq \ldots \subsetneqq J_1$.

{\rm (2) } If  $\mathbb{A}$ is a $B$ tensor, then either $\mathbb{A}$ is a strictly diagonally dominated  $Z$
tensor itself, or we have (\ref{eq42})  where $\mathbb{M}$ is a strictly diagonally dominated  $Z$ tensor.
\end{theorem}

\begin{proof}
Now we show (1) holds.   Define $J(\mathbb{A}) \subseteq [n]$ as
$$J(\mathbb{A}) = \{i \in [n] : \mbox{ there is at least one positive off-diagonal entry in the $i$th row of }   \mathbb{A}\}.$$

\noindent {\bf Case 1: } $J(\mathbb{A})=\emptyset$.

Then $\mathbb{A}$ is a $Z$ tensor, thus a diagonally dominated  $Z$ tensor by Theorem \ref{thm33}. The result holds.

\noindent {\bf Case 2: } $J(\mathbb{A})\not=\emptyset$.

 Let $\mathbb{A}_1 =\mathbb{A}$. For each $i\in J(\mathbb{A}_1)=J_1$, let $d_i$ be the value of the largest off-diagonal entry in the $i$th row of $\mathbb{A}_1$. Let
$$h_1 = \min\{d_i : i\in J_1\}, \quad |J_1|=r\le n.$$
Then $h_1 > 0$.

By the definitions of $\mathbb{A}_1$ and $J_1$,  we know that for the indices $i, i_2, \ldots, i_m$ with
$a_{ii_2\ldots i_m}>0$ and $\delta_{ii_2\ldots i_m}=0$, we have
\begin{equation*}\label{eq43}
i, i_2, \ldots, i_m\in J_1.\end{equation*}
This implies that for the indices $i_1, i_2, \ldots, i_m$ with $\delta_{i_1i_2\ldots i_m}=0$,
if there exists some $j\in[m]$ such that $i_j\not\in J_1$, then $a_{i_1 i_2 \ldots i_m}\le0$.

Now we consider
$$\mathbb{A}_2 = \mathbb{A}_1 -h_1\mathcal{E}^{J_1}=(a_{i_1i_2\ldots i_m}^{(2)}).$$
 It is clear that $a_{i_1\ldots i_m}^{(2)}=a_{\sigma(i_1)\ldots \sigma(i_m)}^{(2)}$ for any $\sigma\in S_m$ whenever $a_{i_1\ldots i_m}^{(2)}>0$ and $\delta_{i_1\ldots i_m}=0$. Now we  show that  $\mathbb{A}_2$ is still a $B_0$ tensor.

  By the definition of $\mathcal{E}^{J_1}$, we have
$$a_{i_1 i_2 \ldots i_m}^{(2)}=\left\{
   \begin{array}{ll}
   a_{i_1 i_2 \ldots i_m}-h_1, & \mbox { if }  i_1,i_2,\ldots,i_m\in J_1;\\
     a_{i_1 i_2 \ldots i_m},  &\mbox { if there exists some } i_j\not\in J_1 \mbox { where }  j\in [m].
   \end{array} \right.$$
Thus for the index $i\in J_1 $, there exist indices $j_2,\ldots, j_m\in J_1$  with $\delta_{ij_2\ldots j_m}=0$
such that $a_{i j_2 \ldots j_m}^{(2)}=a_{i j_2 \ldots j_m}-h_1\geq 0$, then by  $\mathbb{A}$ is a $B_0$ tensor, we have

$\frac{1}{n^{m-1}}\left(\sum\limits_{i_2,\ldots, i_m=1}^na_{i i_2 \ldots i_m}^{(2)}\right)$

\noindent\hskip.2cm $=\frac{1}{n^{m-1}}\left(\sum\limits_{i_2,\ldots, i_m\in J_1}(a_{i i_2 \ldots i_m}-h_1)+
\sum\limits_{\mbox { there exists some } i_j\not\in J_1 \mbox { where } 2\leq j\leq m}a_{i i_2 \ldots i_m}\right)$

\noindent\hskip.2cm $=\frac{1}{n^{m-1}}\left(\sum\limits_{i_2,\ldots, i_m=1}^na_{i i_2 \ldots i_m}\right)-\frac{r^{m-1}}{n^{m-1}}h_1$
\vskip.2cm
\noindent\hskip.2cm $\ge \max\{a_{i j_2 \ldots j_m}: \delta_{ij_2\ldots j_m}=0\}-h_1$
\vskip.2cm
\noindent\hskip.2cm $=\max\{a_{i j_2 \ldots j_m}^{(2)}: \delta_{ij_2\ldots j_m}=0\}$
\vskip.2cm
\noindent\hskip.2cm $\geq 0$,
\vskip.2cm
\noindent and for the other index $i\not\in J_1 $, we know for any indices $j_2,\ldots, j_m\in [n]$  with $\delta_{ij_2\ldots j_m}=0$,
  $a_{i j_2 \ldots j_m}^{(2)}=a_{i j_2 \ldots j_m}\leq 0$ and $a_{i i \ldots i}^{(2)}=a_{i i \ldots i}$, then by  $\mathbb{A}$ is a $B_0$ tensor, we have
$$\sum\limits_{i_2,\ldots, i_m=1}^na_{i i_2 \ldots i_m}^{(2)}=\sum\limits_{i_2,\ldots, i_m=1}^na_{i i_2 \ldots i_m}\geq 0,$$

\noindent and
\vskip.2cm
 $\frac{1}{n^{m-1}}\left(\sum\limits_{i_2,\ldots, i_m=1}^na_{i i_2 \ldots i_m}^{(2)}\right)$

\noindent\hskip.2cm $=\frac{1}{n^{m-1}}\left(\sum\limits_{i_2,\ldots, i_m=1}^na_{i i_2 \ldots i_m}\right)$
\vskip.2cm
\noindent\hskip.2cm $\ge \max\{a_{i j_2 \ldots j_m}: \delta_{ij_2\ldots j_m}=0\}$
\vskip.2cm
\noindent\hskip.2cm $=\max\{a_{i j_2 \ldots j_m}^{(2)}: \delta_{ij_2\ldots j_m}=0\}.$
\vskip.2cm

Therefore  $\mathbb{A}_2=(a_{i_1\ldots i_m}^{(2)})$ is still a $B_0$ tensor such that $a_{i_1\ldots i_m}^{(2)}=a_{\sigma(i_1)\ldots \sigma(i_m)}^{(2)}$ for any $\sigma\in S_m$ whenever $a_{i_1\ldots i_m}^{(2)}>0$ and $\delta_{i_1\ldots i_m}=0$.

We now replace $\mathbb{A}_1$ by $\mathbb{A}_2$, and repeat this process. We see that
$$J_2=J(\mathbb{A}_2) = \{i \in [n] : \mbox{ there is at least one positive off-diagonal entry in the $i$th row of }  \mathbb{A}_2\}.$$
is a proper subset of $ J(\mathbb{A}_1)=J_1$. Repeat this process until $J_{s+1}=J(\mathbb{A}_{s+1})=\emptyset$. Let $\mathbb{M}=\mathbb{A}_{s+1}$.
We see that (\ref{eq42}) holds. It is obvious that $s$ is a positive integer, $h_k > 0$ and
$J_k\subseteq [n]$, for $k = 1, \ldots, s$, and $J_s\subsetneqq J_{s-1}\subsetneqq \ldots \subsetneqq J_1$. This proves (1).

The proof of (2) is similar to the proof of (1), so we omit it.
\end{proof}

By Theorem \ref{thm43} and the result (2) of Theorem \ref{thm13}, we can obtain the following results immediately.

\begin{corollary}{\rm (\cite{SQ2}, Theorem 4)}

{\rm(1) } An even order symmetric $B_0$ tensor  is a $P_0$ tensor and positive semi-definite.

{\rm(2) } An even order symmetric $B$ tensor is a $P$ tensor and positive definite.
\end{corollary}
\begin{proof}
Now we only show (1) holds, the proof of (2) is similar, we omit it.

Let $m$ be even and $\mathbb{A}=(a_{i_1\ldots i_m})$ be a symmetric $B_0$ tensor.
By Theorem \ref{thm43}, if $\mathbb{A}$ itself is a diagonally dominated symmetric $Z$ tensor,
then it is  a $P_0$ tensor and positive semi-definite by Theorem \ref{thm36}.
Otherwise, (\ref{eq42}) holds with $s > 0$. Let $x=(x_1, \ldots, x_n)^T \in\mathfrak{R}^n$. Then by (\ref{eq41}), (\ref{eq42}) and
$\mathbb{M}$ is a diagonally dominated symmetric $Z$ tensor, we have
$$\mathbb{B}x^m =\mathbb{M}x^m+\sum\limits_{k=1}^s h_k\mathcal{E}^{J_k}x^m=\mathbb{M}x^m+\sum\limits_{k=1}^s h_k\left(\sum\{x_j : j\in J_k\}\right)^m\ge0.$$
 This implies (1) holds by the definition of positive semi-definite and result (2) of Theorem \ref{thm13}.
\end{proof}

\end{document}